\documentclass[12pt]{article}

\usepackage{hyperref}

\usepackage{amssymb,amsmath,amsthm}
\usepackage[longnamesfirst]{natbib}
\bibpunct[ ]{(}{)}{,}{a}{}{,}

\theoremstyle{definition}
\newtheorem{example}{Example}

\theoremstyle{plain}

\newtheorem{lemma}[example]{Lemma}
\newtheorem{proposition}[example]{Proposition}
\newtheorem{theorem}[example]{Theorem}

\theoremstyle{remark}

\newcommand{\bE}{\mathbf{E}}

\newcommand{\bU}{\mathbf{U}}
\newcommand{\bD}{\mathbf{D}}

\newcommand{\bQ}{\mathbf{Q}}
\newcommand{\bT}{\mathbf{T}}

\newcommand{\be}{\mathbf{e}}

\newcommand{\bI}{\mathbf{I}}
\newcommand{\bzero}{\mathbf{0}}
\newcommand{\bC}{\mathbf{C}}
\newcommand{\bP}{\mathbf{P}}
\newcommand{\bB}{\mathbf{B}}
\newcommand{\bA}{\mathbf{A}}

\newcommand{\bu}{\mathbf{u}}

\newcommand{\eps}{\epsilon}
\newcommand{\bv}{\mathbf{v}}

\newcommand{\E}{\mathbb E}

\begin{document}

\title{Absolute continuity of the limiting eigenvalue distribution of the random Toeplitz matrix}
\author{Arnab Sen and B\'{a}lint Vir\'{a}g}

\date{}

\maketitle

\begin{abstract}
We show that the limiting eigenvalue distribution of random symmetric Toeplitz matrices is absolutely continuous with density bounded by $8$,  partially answering a question of \cite{Bryc06}. The main tool used in the proof is a spectral averaging technique from the theory of random Schr\"{o}dinger operators.
The similar question for Hankel matrices remains open.
\end{abstract}

\section{Introduction}

An $n \times n$ symmetric random  Toeplitz matrix is given by $$ \bT_{n} = ((a_{|j - k|}))_{ 0 \le  j, k \le n}$$
where $(a_j)_{  j \ge 0}$ is a sequence of
i.i.d.\ random variables with $\mathrm{Var}(a_0)=1$. For a $m \times m$ Hermitian matrix $\bA$, we denote by 
$$\mu(\bA): = \frac{1}{m} \sum_{i=1}^m \delta_{\lambda_i }$$ 
the empirical eigenvalue distribution of $\bA$, where $\lambda_j, 1 \le j \le m$  
are the eigenvalues of  $\bA$, counting multiplicity.
\citet*{Bryc06} established using method of moments that  with probability $1$, $\mu(n^{-1/2}\bT_n)$  converges weakly as $n \to \infty$ to a nonrandom symmetric probability measure $\gamma$ which does not depend on the distribution of $a_0$, and has unbounded support. They conjecture (see Remark 1.1 there) that $\gamma$ has a smooth density.
In this note, we give a partial solution:

\begin{theorem}\label{thm:main}
The measure $\gamma$ is absolutely continuous with density bounded by $8$.
\end{theorem}
The actual bound we get is  $\frac{16\sqrt 2}{\pi} = 7.20 \ldots$, but we do not expect it to be optimal.

It seems that the method of moments is of little use in determining the existence of the absolute continuity of the limiting eigenvalue distribution.  Indeed our proof goes along a completely different path. We make use of the fact that the spectrum of the Gaussian Toeplitz matrix can be realized as that of some diagonal matrix consisting of  independent Gaussians   conjugated by an appropriate projection matrix -  a fact observed in a recent paper \cite{SenVirag11}.   The next key ingredient of our proof is a spectral averaging technique
(Wegner type estimate) developed by \cite{Combes96} in connection to the problem of localization for certain families of random Schr\"{o}dinger operators.

Our proof does not establish further smoothness property of $\gamma$. The absolute continuity  for the limiting distribution of random Hankel matrices also remains open.

\section{Connection between Toeplitz and circulant matrices}
Since $\gamma$ does not depend on the distribution of $a_0$, from now on, we will assume, without any loss, that $(a_i)_{i \ge 0}$ are i.i.d.\  standard Gaussian random variables.  The remainder of the section we recall some facts about the connection between Toeplitz matrices and circulant matrices from \cite{SenVirag11}. Let $\bT_n^\circ$ be the symmetric Toeplitz matrix  which has  $\sqrt{2} a_0$ on its diagonal instead of $a_0$. It can be easily shown (e.g.\ using  Hoffman-Wielandt inequality, see \cite{Bhatia97}) that this modification has no effect as far as  the limiting eigenvalue distribution is concerned.

 $\bT^\circ_n$  is the $n \times n$ principal
submatrix of a $2n \times 2n$  circulant matrix $\bC_{2n}=
(b_{j-i \;\mathrm{ mod }\;2n})_{0 \le i, j \le 2n-1}$,
where $b_j= a_j$ for $ 0 < j < n$ and $b_j  = a_{2n -j}$
for $n<  j < 2n$,  $b_ 0 = \sqrt 2 a_0, b_n = \sqrt 2 a_n$. In other words,
\begin{equation}\label{eq:circulant_toeplitz}
 \bQ_{2n} \bC_{2n} \bQ_{2n}=  \begin{pmatrix}
\bT_n^\circ & \bzero_n\\
\bzero_n & \bzero_n
\end{pmatrix},
\quad \text{ where }
\bQ_{2n} =
\begin{pmatrix}
\bI_{n} & \bzero_n\\
\bzero_n & \bzero_n
\end{pmatrix}.
\end{equation}
The circulant matrix can be easily diagonalized as $  (2n)^{-1/2}\bC_{2n} = \bU_{2n} \bD_{2n} \bU_{2n}^*$
where  $\bU_{2n}$ is the discrete Fourier transform, i.e.\
a unitary matrix given by
\[\bU_{2n}(j, k)  = \frac{1}{\sqrt{2n}}\exp \left(  \frac{2 \pi i jk}{2n}\right),  0 \le j, k \le 2n-1\] and $\bD_{2n} = \mathrm{diag}(d_0, d_1, \ldots, d_{2n-1})$, where
\[ d_j = \frac{1}{\sqrt{2n}} \sum_{k=0}^{2n-1} b_k \exp \left(  \frac{2 \pi i  jk}{2n}\right) = \frac{1}{\sqrt{2n}} \left [ \sqrt 2 a_0 + (-1)^n \sqrt 2 a_n + 2\sum_{k=1}^{n-1} a_k \cos \left( \frac{2 \pi  jk}{2n}\right) \right].   \]
Clearly, $d_j = d_{2n-j}$ for all $n < j < 2n$. Also, $(d_j)_{ 0 \le j \le n}$ are independent mean zero Gaussian random variables with $\mathrm{Var}(d_j)=1$ if $ 0< j< n$ and  $\mathrm{Var}(d_j)=2$ if $  j \in \{0, n\}$. Define
$\bP_{2n} := \bU_{2n}^*
\bQ_{2n} \bU_{2n}$ so that
 \begin{equation} \label{eq:PDP}
 (2n)^{-1/2} \bU^*_{2n} \bQ_{2n} \bC_{2n} \bQ_{2n}  \bU_{2n}=  \bP_{2n} \bD_{2n} \bP_{2n}.
 \end{equation}
 Check that
$\bP_{2n}$  is a Hermitian projection matrix with $\bP_{2n}(j, j)=1/2$ for all $j$.  For notational simplification, we will drop the
subscript ${2n}$ from the relevant matrices unless we want
to emphasize the dependence on $n$.

\section{Proof of the main theorem}

For a vector $\bu \in \mathbb C^m$, let $\sigma(\bA, \bu)$ be the spectral measure of matrix $\bA$ at $\bu$. For a finite measure $\nu$ on $\mathbb R$, its Cauchy-Stieltjes transform is given by
 \[ s(z; \nu) = \int_{\mathbb R} \frac{1}{x-z} \nu(dx), \quad z \in \mathbb C, \mathrm{Im}(z) > 0. \]
Let $\E \mu(n^{-1/2}\bT_n^\circ)$ denote the expected empirical eigenvalue distribution of $n^{-1/2}\bT_n^\circ$ which is defined by
  $\E \mu(n^{-1/2}\bT_n^\circ)(B) =  \E[ \mu(n^{-1/2}\bT_n^\circ)(B)]$ for all Borel sets $B$.

\begin{lemma} \label{l:toep_stieltjes}
Let $(\be_j)_{0 \le j\le 2n-1}$ be the coordinate vectors of $\mathbb R^{2n}$. Then
  \[ s(z; \E\mu(n^{-1/2}\bT_n^\circ)) = \frac{\sqrt 2}{n} \sum_{j=0}^{2n-1} \E \langle \bP \be_j,  (  \bP\bD\bP - z\bI)^{-1} \bP\be_j \rangle \quad z \in \mathbb C, \mathrm{Im}(z) > 0. \]
\end{lemma}
Before we start proving the above lemma, we state a simple fact about spectral measures of Hermitian matrices without proof.
\begin{lemma}\label{spectral_equality}
Let $\bA$  be an $m \times m$ Hermitian matrix. Let $\bu_1, \bu_2, \ldots, \bu_k$ and $\bv_1, \bv_2, \ldots, \bv_\ell$ be vectors in $\mathbb C^m$ satisfying $\sum_{i=1}^k \bu_i \bu_i^* = \sum_{j=1}^\ell \bv_j \bv_j^*$. Then
\[ \sum_{i=1}^k \sigma( \bA, \bu_i) = \sum_{j=1}^\ell \sigma( \bA, \bv_j). \]
\end{lemma}
\begin{proof}[Proof of Lemma~\ref{l:toep_stieltjes}]
By \eqref{eq:circulant_toeplitz}, we have
\[ s(z;\mu(n^{-1/2}\bT_n^\circ)) = \frac1n \sum_{j=0}^{n-1} \langle \be_j,  ( n^{-1/2} \bQ\bC\bQ - z\bI)^{-1} \be_j \rangle,  \]
Changing basis as in \eqref{eq:PDP}, we can rewrite this as \[\frac{\sqrt 2}{n} \sum_{j=0}^{n-1} \langle \bU^* \be_j,  (  \bP\bD\bP - z\bI)^{-1} \bU^*\be_j \rangle = \frac{\sqrt 2}{n} \sum_{j=0}^{n-1} s(z;\sigma( \bP\bD\bP,  \bU^* \be_j)).  \]
Now by Lemma~\ref{spectral_equality} and  the fact that $\sum_{j=0}^{n-1} \bU^* \be_j \be_j^* \bU  = \sum_{j=0}^{2n-1} \bP \be_j \be_j^* \bP $, we deduce
 \begin{equation}\label{eq: s_equality_toep}
  s(z;\mu(n^{-1/2}\bT_n^\circ)) = \frac{\sqrt 2}{n} \sum_{j=0}^{2n-1} \langle \bP \be_j,  (  \bP\bD\bP - z\bI)^{-1} \bP\be_j \rangle.
  \end{equation}
 The lemma now follows by taking expectation on both sides of \eqref{eq: s_equality_toep} and by observing that for a fixed $z \in \mathbb C, \mathrm{Im}(z) \ne 0$, the map $\nu \mapsto s(z; \nu)$ is linear and hence commutes with the expectation.
 \end{proof}
Next we will prove a key lemma about the uniform bound on the Stieltjes transform of the expected empirical eigenvalue distribution of Toeplitz matrices.
\begin{lemma}  \label{lem:key_bound}
For all $n$, we have
\[   \sup_{z:\mathrm{Im}(z)>0} |s(z; \E\mu(n^{-1/2}\bT_n^\circ)) |  \le 16 \sqrt 2.\]
\end{lemma}
The above lemma will be a direct consequence of the following result of \cite{Combes96} on  the spectral averaging for one parameter family self-adjoining operators.

\begin{proposition}[\cite{Combes96}] \label{thm:spec_avg}
Let $H_\lambda, \lambda \in \mathbb R $  be a  $C^2$-family of self-adjoint operators
such that  $D(H_\lambda) =  D_0 \subset \mathcal H \  \forall \lambda \in \mathbb R$, and such that  $(H_\lambda - z)^{-1} $ is twice strongly differentiable in $\lambda$ for all  $z,   \mathrm {Im} (z) \ne 0$.
 Assume that there exist a  finite positive constant $c_0$, and a positive bounded self-adjoint operator $B$ such that, on $D_0$
 \begin{align}
 \dot{H}_\lambda &:= \frac{d H_\lambda}{d\lambda} \ge c_0 B^2. \label{assump1}
 \end{align}
 Also assume $H_\lambda$ is linear in $\lambda$, i.e., $ \ddot{H}_\lambda := \frac{d^2 H_\lambda}{d\lambda^2} = 0$.
   Then for all $E \in \mathbb R$ and twice continuously differentiable function $g$ such that $g, g', g'' \in L^1(\mathbb R)$ and for all $\varphi \in \mathcal H$,
\begin{equation} \label{sbound}
\sup_{\delta >0} \left | \int_{\mathbb R} g(\lambda) \langle \varphi, B(H_\lambda  - E - i\delta)^{-1} B\varphi \rangle d\lambda \right| \le c_0^{-1} (\| g \|_1  + \| g'\|_1 + \|g''\|_1)\| \varphi\|^2.
\end{equation}
\end{proposition}
 Proposition~\ref{thm:spec_avg}  is an immediate corollary of Theorem~1.1 of \cite{Combes96} where instead of  $\ddot{H}_\lambda = 0$, it was assumed that $ |\ddot{H}_\lambda| \le c_1 \dot{H}_\lambda$.  The vanishing second derivative assumption shortens the the proof by a considerable amount.  We have  included a proof of the above proposition in the appendix to make this paper self-contained and also to make constant in the bound \eqref{sbound} explicit.


\begin{proof}[Proof of Lemma~\ref{lem:key_bound}] Set $\bE_j  =  \be_j\be_j^*  +  \be_{2n-j}\be_{2n-j}^*$   for $1\le j < n$,  and $\bE_j  =  \be_j\be_j^* $ for $j \in \{0, n\}$.  Take
\begin{equation} \label{B_bound}
\bB_j = \bP  \be_j\be_j^* \bP \text{ or } \bP \be_{2n-j}\be_{2n-j}^* \bP  \text{ for }  1 \le j < n \text{ and } \bB_j = \bP  \be_j\be_j^* \bP \text{  for } j \in \{0, n\}.
\end{equation}
Fix $0 \le j \le n$. We apply Theorem ~\ref{thm:spec_avg} with  $H_\lambda= \bP \big( \bD + ( \lambda -d_j) \bE_j  \big) \bP$. In words, we replace $d_j$ and $d_{2n-j}$  by $\lambda$ in  $\bP\bD\bP$ to get $H_\lambda$. Note that $H_\lambda$ is random self-adjoint operator  which is a function of $\{ d_k: 0 \le k \le n,  k \ne j\}$. Also, $H_\lambda$ is linear in $\lambda$ and so,  $\ddot{H}_{\lambda}  = 0$.
Since $\dot{H}_{\lambda}  =   \bP \bE_j \bP \ge \bB_j  = \bP(j, j)^{-1} \bB_j^2$, the condition \eqref{assump1} is satisfied with $B = \bB_j$ and $c_0=2$ since $\bP(j, j)=1/2$.
  Take $g = \phi_j$ where   $\phi_j$ be the density of $Z$ for $0< j<n$ or  the density of $\sqrt 2 Z$ for $j \in \{0, n\}$,  $Z$ being a standard Gaussian random variable. It is easy to check that $\|g\|_1=1, \|g'\|_1\le \sqrt{\frac{2}{\pi}}, \| g''\|_1 \le 2 $.
Then plugging $\varphi = \be_j$ or $\be_{2n-j}$ and $\bB_j = \bP  \be_j\be_j^* \bP$  or $\bP \be_{2n-j}\be_{2n-j}^* \bP $  in \eqref{sbound} and taking expectation w.r.t.\ the remaining randomness $\{ d_k: 0\le k \le n, k\ne j\}$, we obtain
\begin{equation}\label{eq:sbound1}
 \sup_{z:\mathrm{Im}(z)>0}  \bP(j, j)^2\Big |  \E \langle \bP \be_j,  (  \bP\bD\bP - z\bI)^{-1} \bP\be_j \rangle \Big | \le c_0^{-1}(\| g \|_1  + \| g'\|_1 + \|g''\|_1) \le 2.
 \end{equation}
The lemma is now immediate from \eqref{eq:sbound1} and Lemma~\ref{l:toep_stieltjes}.
\end{proof}

\begin{proof}[Proof of Theorem~\ref{thm:main}]
It follows from  the inversion formula, $\nu \{ (x, y)\} = \lim_{ \delta \downarrow 0} \frac{1}{\pi} \int_x^y \mathrm{Im} (s(E+ i \delta; \nu)) dE$ for all $x<y$ continuity points of $\nu$ that
if for some probability measure $\mu$,  $\sup_{z:\mathrm{Im}(z)>0} \mathrm{Im} (s(z; \mu ) ) \le K $
then $\mu$ is absolutely continuous w.r.t.\  Lebesgue measure and its density is bounded by $\pi^{-1}K$.

Note that $  s(z; \E \mu(n^{-1/2}\bT_n^\circ))  \to s(z; \gamma)$ as $n \to \infty$ for each $z \in \mathbb C, \mathrm{Im}(z)>0$ since $ \E \mu(n^{-1/2}\bT_n^\circ)$ converges weakly to $\gamma$ (see \cite{Bryc06}).
So  by Lemma~\ref{lem:key_bound}, it follows that
\[  \sup_{z:\mathrm{Im}(z)>0} |s(z; \gamma ) |  \le 16 \sqrt{2}<8\pi\]
which completes the proof of the theorem.
\end{proof}

\section*{Appendix}
\begin{proof}[Proof of Proposition~\ref{thm:spec_avg}]
Define for $ \eps >0$ and $0 < \delta < 1$,
\begin{equation} \label{def:R}
R(\lambda, \eps, \delta) := (H_\lambda - E + i\delta + i\eps \dot{H}_\lambda)^{-1}
\end{equation}
and set
\begin{equation} \label{def:K}
K(\lambda, \eps, \delta) :=  B R(\lambda, \eps, \delta) B.
\end{equation}
Note that from assumption \eqref{assump1} ,
\[ - \mathrm{Im} \langle \varphi, K(\lambda, \eps, \delta)  \varphi \rangle = \langle \varphi, B R(\lambda, \eps, \delta)^*( \delta + \eps \dot{H}_\lambda)R(\lambda, \eps, \delta)B \varphi \rangle  \ge c_0 \eps \| K(\lambda, \eps, \delta) \varphi \|^2,  \]
which,  coupled with  Cauchy-Schwarz inequality, implies that $\forall \varphi \in \mathcal {H}, \|\varphi \|=1$,
\begin{equation} \label{ineq:apriori_bound}
 \| K(\lambda, \eps, \delta) \varphi \| \ge -  \mathrm{Im} \langle \varphi, K(\lambda, \eps, \delta)  \varphi \rangle \ge  c_0 \eps \| K(\lambda, \eps, \delta) \varphi \|^2.
 \end{equation}
Now define
\[  F(\eps, \delta) :=  \int_{\mathbb R} g(\lambda)  \langle \varphi, K(\lambda, \eps, \delta)  \varphi \rangle d\lambda.\]
Inequality \eqref{ineq:apriori_bound} implies the bound
\begin{equation} \label{F_bound}
 F(\eps, \delta) \le (\eps c_0)^{-1} \| g\|_1.
\end{equation}
Now differentiating $F$ w.r.t.\ $\eps$, we obtain
\begin{align*}
i \frac{ dF(\eps, \delta)}{d \eps}  &=   \int_{\mathbb R} g(\lambda)  \langle \varphi,  B R(\lambda, \eps, \delta)  \dot{H}_\lambda R(\lambda, \eps, \delta) B  \varphi \rangle d\lambda\\
&=  - \int_{\mathbb R} g(\lambda)  \frac{d}{d\lambda} \langle \varphi,  K(\lambda, \eps, \delta)  \varphi \rangle d\lambda.
\end{align*}
where the last equality follows from the fact $\ddot{H_\lambda} = 0$. Therefore, from \eqref{ineq:apriori_bound} and  by integration of parts,
\begin{equation}\label{F_deriv_bound}
 \left | \frac{ dF(\eps, \delta)}{d \eps} \right|  = \left |\int_{\mathbb R} g'(\lambda)  \langle \varphi,  K(\lambda, \eps, \delta)  \varphi \rangle d\lambda \right | \le  (\eps c_0)^{-1} \| g' \|_1.
 \end{equation}
By integrating the differential inequality \eqref{F_deriv_bound} and using the bound \eqref{F_bound}, we  can improve the bound for $F$ as
\begin{equation}\label{F_improved_bound}
 |F(\eps, \delta) | \le  c_0^{-1} \| g' \|_1  \cdot | \log \eps | + | F(1, \delta)|  \le  c_0^{-1}\| g' \|_1 \cdot | \log \eps | + c_0^{-1} \| g\|_1,  \ \ \forall \eps \in (0, 1).
  \end{equation}
Now if we consider  the function $\tilde F(\eps, \delta) :=  \int_{\mathbb R} g'(\lambda)  \langle \varphi, K(\lambda, \eps, \delta)  \varphi \rangle d\lambda,$ then by replacing the function $g$ by its derivative $g'$ in \eqref{F_improved_bound}, we deduce that
\[   |\tilde F(\eps, \delta) | \le    c_0^{-1}\| g'' \|_1 \cdot | \log \eps | + c_0^{-1} \| g'\|_1,   \ \ \forall \eps \in (0, 1) \]
which further implies that
\begin{equation}\label{F_deriv_bound_new}
 \left | \frac{ dF(\eps, \delta)}{d \eps} \right| \le  c_0^{-1}\| g'' \|_1 \cdot | \log \eps | + c_0^{-1} \| g'\|_1,  \ \ \forall \eps \in (0, 1).
 \end{equation}
Again integrating \eqref{F_deriv_bound_new}, we get
\begin{equation} \label{F_ultimate_bound}
  | F(\eps, \delta) | \le c_0^{-1} (\| g'' \|_1  + \| g'\|_1) + | F(1, \delta)| \le c_0^{-1} (\| g'' \|_1  + \| g'\|_1 + \|g\|_1),
  \end{equation}
which holds for all $\eps, \delta \in (0, 1)$. The proof of the Proposition now follows from the fact that $R(\lambda, \eps, \delta)$  converges weakly to $(H_\lambda - E +i \delta)^{-1}$ as  $\eps \to 0+$ provided $\delta>0$, and the dominated convergence theorem since
$ \left | \int_{\mathbb R} g(\lambda)  \langle \varphi, K(\lambda, \eps, \delta)  \varphi \rangle d\lambda \right | \le C$,
by \eqref{F_ultimate_bound}.
\end{proof}

\section*{Acknowledgments}
 Most of the research has been conducted while A.S. visited  Technical University of Budapest in 
July 2011. A.S.  is supported by  the EPSRC grant  EP/G055068/1.

\bibliographystyle{dcu}
\bibliography{toeplitz}

\bigskip\bigskip\bigskip\bigskip\noindent
\begin{minipage}{0.49\linewidth}
Arnab Sen \\
Statistical Laboratory\\
Center for Mathematical Sciences\\
University of Cambridge\\
Wilberforce Road\\
Cambridge, UK CB3 0WB\\
\tt{a.sen@statslab.cam.ac.uk}
\end{minipage}
\begin{minipage}{0.49\linewidth}
B\'alint Vir\'ag
\\Departments of Mathematics \\and Statistics
\\University of Toronto
\\Toronto ON~~
\\M5S 2E4, Canada
\\{\tt balint@math.toronto.edu}
\end{minipage}

\end{document}